\documentclass[preprint,12pt]{elsarticle}

\usepackage{amsmath}
\usepackage{amssymb}
\usepackage{amsthm}
\usepackage[mathscr]{eucal}
\usepackage{here}

\newcommand{\rint}{\mathbb{Z}}
\newcommand{\rat}{\mathbb{Q}}
\newcommand{\real}{\mathbb{R}}
\newcommand{\complex}{\mathbb{C}}

\newcommand{\plog}{\mathrm{Li}}
\newcommand{\wt}[1]{\widetilde{#1}}
\newcommand{\cang}{\varphi}
\newcommand{\cutz}{{\ooalign{$\multimap$\crcr{\hspace{1.66ex}\raisebox{.2ex}{\scriptsize $\bullet$}}}}}
\newcommand{\cut}[1]{\ell_{#1}}
\newcommand{\cutsym}[1]{\ell_{#1}^{\,\mathrm{sym}}}

\newcommand{\dom}[1]{(#1)\;\;}
\newcommand{\rnums}[1]{\mathrm{\rnum{#1}})\;\;}
\newcommand{\midflex}{\;\middle|\;}
\newcommand{\closure}[1]{\overline{#1}}

\makeatletter

\@addtoreset{equation}{section}
\makeatother
\newtheorem{theorem}{Theorem}[section]
\newtheorem{proposition}[theorem]{Proposition}
\newtheorem{corollary}[theorem]{Corollary}
\newtheorem{lemma}[theorem]{Lemma}
\newtheorem{rem}[theorem]{Remark}
\newenvironment{remark}{\begin{rem}\normalfont}{\end{rem}}

\def\rnum#1{\expandafter{\romannumeral #1}}
\def\Rnum#1{\uppercase\expandafter{\romannumeral #1}}

\begin{document}

\begin{frontmatter}



\title{An explicit expression of the Lerch zeta function on maximal domains of holomorphy}


\author{Rintaro Kozuma}

\affiliation{organization={College of International Management, Ritsumeikan Asia Pacific University},
            city={Beppu-shi},
            state={Oita},
            postcode={874-8577},
            country={Japan}}

\begin{abstract}
We give two results on the Lerch zeta function $\Phi(z,\,s,\,w)$.
The first is to give an explicit expression providing both the analytic continuation of $\Phi$
 in $n$-variables $(n \in \{1,\,2,\,3\})$ to maximal domains of holomorphy in $\complex^n$
 with computable evaluation
 and an extended formula for the special values of $\Phi$ at non-positive integers in the variable $s$.
The second is to show that Lerch's functional equation is essentially the same
 as Apostol's functional equation using the first result.
\end{abstract}



\begin{keyword}
Lerch zeta function \sep analytic continuation \sep functional equation

\MSC[2020] 11M35
\end{keyword}

\end{frontmatter}


\section{Introduction}

The Lerch zeta function (sometimes called the Lerch transcendent)
\begin{align}\label{def:init}
\Phi(z,\,s,\,w)=\sum_{n=0}^{\infty} \frac{z^n}{(n+w)^s},
\end{align}
 introduced by Lerch~\cite{Ler} and Lipschitz~\cite{Lip} in the case $z \in \complex,\,|z| \leq 1$, is a generalization of the Hurwitz zeta function
$$\zeta(s,\,w)=\sum_{n=0}^{\infty} \frac{1}{(n+w)^s}$$
and the polylogarithm
$$\plog_s(z)=\sum_{n=1}^{\infty} \frac{z^n}{n^s}.$$
One can verify that $\Phi$ converges absolutely in either the case
\begin{align}\label{absconv}
\text{$|z|<1$, $s \in \complex$, $w \in \complex\setminus\rint_{\leq 0}$,
 or $|z|=1$, $\Re(s)>1$, $w \in \complex\setminus\rint_{\leq 0}$.}
\end{align}
Many individual works on analytic continuation of $\Phi(z,\,s,\,w)$ in some of the variables $z,\,s,\,w$
 with restricted condition
 have been done using various integral representation technics.
Lerch~\cite{Ler} obtained the analytic continuation to the whole $s$-plane
 with the condition $|z| \leq 1,\,\Re(w)>0$ according to the line of Riemann's classic method.
Berndt~\cite{Ber} gave a simple proof for the analytic continuation in $s$ with the condition $z=e^{2\pi ia}\,(0<a<1),\,0<w \leq 1$
 using an original contour integral representation.
Ferreira and L\'{o}pez~\cite[Proposition~1]{Fer-Lop} showed the analytic continuation in $s$ with the condition
 $z \in \Omega_a$, $w \in \complex\setminus\real_{<0}$, where $\Omega_a$ is a proper subset of $\complex$
 depending on the sign of $\Re(w)$.
Also, as studied by Erd\'{e}lyi~\cite[\S1.11]{HTF}, the integral representation given by Lerch
 can be regarded as an analytic function of $z$ which is holomorphic in a suitable cut plane.
On the other hand,
 Kanemitsu, Katsurada and Yoshimoto~\cite{KKY}, and
 Lagarias and Li~\cite{LL2}, \cite{LL3} considered $\Phi(z,\,s,\,w)$ as an analytic function
 in three complex variables $(z,\,s,\,w)$ and gave analytic continuations
 to various large domains in $\complex^3$.
Especially, Lagarias and Li obtained an analytic continuation to a multivalued function on
 a maximal domain of holomorphy in three variables~\cite[Theorem~3.6]{LL3}.

In this paper,
for a set of complex variables $V$ on a domain $D \subset \complex^n$,
 we say {\it holomorphic on $D$ in $V$} or simply {\it holomorphic on $D$}
 by means of the holomorphy of multivariable(one variable) complex functions.
 {\it i.e.} $f$ is holomorphic with respect to each variable in $V$ on $D$
 with other fixed variables.
Also, we fix a real $\cang \in \real\setminus2\pi\rint$ ({\it i.e.} $\Re(e^{i\cang})<1$)
 and choose the principal value of the argument (function) $\arg \lambda$ of a complex $\lambda \in \complex^*$
 by $\cang \leq \arg \lambda<\cang+2\pi$, which defines the principal branch of the logarithm function $\log\lambda$,
 forces $(n+w)^s=e^{s\log(n+w)}$ to be single-valued,
 and $\log r$ is not always real for $r \in \real_{>0}$; namely,
 there exists a unique integer $\nu$ (depending on the choice $\cang$) such that
$$\log r=\mathrm{ln}\,r+2\pi i \nu \quad \text{ for any $r \in \real_{>0}$},$$
 where $\mathrm{ln}\,r\,(r \in \real_{>0})$ denotes the usual logarithm valued in $\real$.
Then the infinite series~(\ref{def:init}) defines a single-valued function in complex multivariable(one variable)
 holomorphic on six simply-connected domains
\begin{align*}\label{initdomains}
\begin{array}{ll}
\dom{1} |z|<1,\,s \in \complex,\,w \in \complex\setminus\cut{\varphi} \text{ in $(z,\,s,\,w)$},\\
\dom{2} |z|<1,\,s \in \complex \text{ in $(z,\,s)$ with fixed $w \in \cut{\varphi}\setminus\rint_{\leq 0}$},\\
\dom{1'} \Re(s)>1,\,w \in \complex\setminus\cut{\varphi} \text{ in $(s,\,w)$ with fixed $z \ne 1,\,|z|=1$},\\
\dom{3} \Re(s)>1,\,w \in \complex\setminus\cut{\varphi} \text{ in $(s,\,w)$ with $z=1$},\\
\dom{2'} \Re(s)>1 \text{ in $s$ with fixed $z \ne 1,\,|z|=1,\,w \in \cut{\varphi}\setminus\rint_{\leq 0}$},\\
\dom{4} \Re(s)>1 \text{ in $s$ with $z=1$ and fixed $w \in \cut{\varphi}\setminus\rint_{\leq 0}$},
\end{array}\tag{$D_0$}
\end{align*}
 where $\cut{\cang}$ denotes the set
\begin{align*}
\{w \in \complex \mid w=-n+re^{i\cang} \text{ for $r \geq 0,\,n \in \rint_{\geq 0}$}\}
\end{align*}
 of infinite copies of the complex half line $\{re^{i\cang} \in \complex \mid r \geq 0\}$ (together with $0$).
\begin{figure}[ht]
\centering
\includegraphics{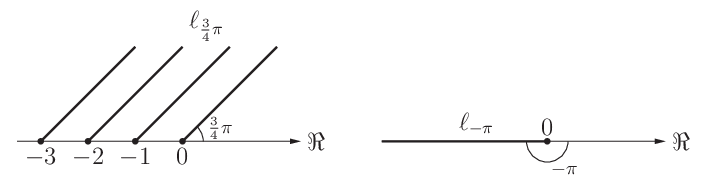}
\caption{\small The lines $\ell_{\cang}$ for $\cang=\frac{3}{4}\pi,\,-\pi$}
\end{figure}
The domains~(\ref{initdomains}) contain the region~(\ref{absconv}),
 and the function~(\ref{def:init}) is discontinuous for all $w \in \cut{\cang}$ with $s \in \complex\setminus\rint$ and all $z$ with $|z|=1$.
Our first result establishes a brief proof for the single-valued analytic continuation of $\Phi$ on (\ref{initdomains}) to
 maximal domains of holomorphy in $\complex^n$ in any $n$-variables in $\{z,\,s,\,w\}$ for each $n \in \{1,\,2,\,3\}$,
 including previous results in \cite{Ler}, \cite{Apo}, \cite{Ber}, \cite{Fer-Lop}, \cite{KKY}, \cite{LL3}.
Specifically, we have the following theorem.

\begin{theorem}\label{thm:explicit}
Let $\epsilon$ be a positive real satisfying
\begin{align}\label{ass1:epsilon}
\Re(e^{i\cang})<\Re(e^{i\arg(1+i\epsilon)})\,(=1/\sqrt{1+\epsilon^2}), \tag{$\epsilon$\,1}
\end{align}
 and let $m,\, N \in \rint_{\geq 0}$.
For each $j \in \{1,\,2,\,3,\,4\}$ $($resp. $j \in \{1,\,2\})$,
 the Lerch zeta function~\eqref{def:init} on the domain~\eqref{initdomains}$(j)$ $($resp. \eqref{initdomains}$(j'))$
 can be analytically continued to the single-valued function 
\begin{align}\label{func:anaconti}
\sum_{n=0}^{N-1}\frac{z^n}{(n+w)^s}&+\frac{z^N}{e^{4\pi i \nu s}\Gamma(s)}\sum_{r=0}^m \frac{\mathscr{B}_r(z,\,N+w)}{r!}\frac{(-1)^r\epsilon^{s+r-1}}{s+r-1} \notag\\
&+\frac{z^N}{e^{4\pi i \nu s}\Gamma(s)}\wt{I}_N+\frac{z^N}{e^{4\pi i \nu s}\Gamma(s)}J_{N,\,m}
\end{align}
holomorphic on the domain~\eqref{domains}$(j)$ $($resp. \eqref{domains}$(j))$, where
\begin{align}\label{domains}
\begin{array}{l}
\dom{1} \wt{D}_z\setminus\cutz_{\cang'} \times \wt{D}_s \times \wt{D}_w\setminus\cut{\cang} \text{ in $(z,\,s,\,w)$},\\
\dom{2} \wt{D}_z\setminus\cutz_{\cang'}\times \wt{D}_s \text{ in $(z,\,s)$ with fixed $w \in (\cut{\cang}\setminus\rint_{\leq 0}) \cap \wt{D}_w$},\\
\dom{3} \wt{D}_s\setminus P_z \times \wt{D}_w\setminus\cut{\cang} \text{ in $(s,\,w)$ with fixed $z \in \cutz_{\cang'} \setminus (1,\,e^{\epsilon}]$},\\
\dom{4} \wt{D}_s\setminus P_z \text{ in $s$ with fixed $z \in \cutz_{\cang'} \setminus (1,\,e^{\epsilon}],\,w \in (\cut{\cang}\setminus\rint_{\leq 0}) \cap \wt{D}_w$},
\end{array}\tag{$\wt{D}$}
\end{align}
\vspace{-1em}
\begin{align*}
\wt{D}_z&:=\left\{z \in \complex \mid z \notin [1,\,e^{\epsilon}]\right\},\\
\wt{D}_s&:=\left\{s \in \complex \mid \Re(s)>-m\right\},\\
\wt{D}_w&:=\left\{w \in \complex \mid \Re(w)>-N\right\},
\end{align*}
and $\cutz_{\cang'}:=\{z \in \complex \mid z=1+re^{i\cang'} \text{ for $r \geq 0$}\}$
 with any fixed real $\cang'$ satisfying $\Re(e^{i\cang'}) \geq 0$,
 and $P_z:=\{1\}$ when $z=1$, or $P_z:=\emptyset$ when $z \ne 1$.
Also, $\mathscr{B}_r(z,\,w)\,(r \in \rint_{\geq 0})$ is Apostol's rational function defined by \eqref{def:Apo},
 and $\wt{I}_N,\,J_{N,\,m}$ are holomorphic functions $($see $\S\ref{subsec:hol}$, $\S\ref{subsec:AI}$, $\S\ref{subsec:MI})$ on the domains~\eqref{domains}.
Furthermore, for any $(z,\,s,\,w) \in \complex \times \wt{D}_s \times \wt{D}_w$,
 if $\epsilon$ satisfies the additional condition both $z \ne e^{\epsilon}$ and
\begin{align}
0 < \epsilon \leq
\begin{cases}
\dfrac{1}{4(|N+w|+1)} & \text{ if $z=1$},\\
\dfrac{|z-1|}{4(|z|+1)(|N+w|+1)} & \text{ if $z \ne 1$},
\end{cases} \tag{$\epsilon$\,2}
\end{align}
 then the following evaluation holds.
\begin{align*}
|\wt{I}_N(z,\,s,\,w)| &\leq
\begin{cases}
\dfrac{1}{\epsilon A_{\epsilon}(z)}E_{\epsilon,\,N}(s,\,w) & \text{if $z \in \complex \setminus [e^{\epsilon},\,\infty)$},\\
\left(1+\frac{1}{\epsilon}\right)\dfrac{C_{\epsilon}(s,\,w)}{B_{\epsilon}(z)}E_{\epsilon,\,N}(s,\,w) & \text{if $z \in (e^{\epsilon},\,\infty)$}.\\
\end{cases}\\
\left|J_{N,\,m}(z,\,s,\,w)\right| &\leq
\begin{cases}
\dfrac{1}{2^m}\dfrac{\epsilon^{\Re(s)-1}}{\Re(s)+m} & \text{ if $z=1$},\\
\dfrac{2^{-m+1}}{|z-1|}\dfrac{\epsilon^{\Re(s)}}{\Re(s)+m} & \text{ if $z \ne 1$}.
\end{cases}
\end{align*}
Here $A_{\epsilon},\,B_{\epsilon},\,C_{\epsilon},\,E_{\epsilon,\,N}$ are defined as \eqref{def:ABC}.
Especially, $\wt{I}_N\,(\text{resp. $J_{N,\,m}$}) \to 0$ as $N\,(\text{resp. $m$}) \to \infty$ on any compact subset of $\complex \times \wt{D}_s \times \wt{D}_w$.
\end{theorem}

From Theorem~\ref{thm:explicit}, taking the limit $\epsilon \to 0$ and $m,\,N \to \infty$ in the expression~\eqref{func:anaconti},
 we have the following theorem
(The special values of $\Phi$ at $s \in \rint_{\leq 0}$ follows from \S\ref{subsec:value}).

\begin{theorem}\label{thm:anaconti}
For each $j \in \{1,\,2,\,3,\,4\}$ $($resp. $j \in \{1,\,2\})$,
 the Lerch zeta function~{\rm (\ref{def:init})} on the domain~$($\ref{initdomains}$)(j)$ $($resp. $($\ref{initdomains}$)(j'))$
 can be analytically continued to a single-valued function holomorphic on the domain~$(D_1)(j)$ $($resp. $(D_1)(j))$, where
\begin{align*}\label{finaldomains}
\begin{array}{ll}
\dom{1} \complex\setminus\cutz_{\cang'} \times \complex \times \complex\setminus\cut{\cang} \text{ in three variables $(z,\,s,\,w)$},\\
\dom{2} \complex\setminus\cutz_{\cang'} \times \complex \text{ in two variables $(z,\,s)$ with fixed $w \in \cut{\cang}\setminus\rint_{\leq 0}$},\\
\dom{3} \complex\setminus P_z \times \complex\setminus\cut{\cang} \text{ in two variables $(s,\,w)$ with fixed $z \in \cutz_{\cang'}$},\\
\dom{4} \complex\setminus P_z \text{ in one variable $s$ with fixed $z \in \cutz_{\cang'},\,w \in \cut{\cang}\setminus\rint_{\leq 0}$}.
\end{array}\tag{$D_1$}
\end{align*}
Especially, for any fixed $(z,\,w) \in \complex \times \complex\setminus\rint_{\leq 0}$,
 The function~{\rm (\ref{def:init})} has the analytic continuation to $\complex\setminus P_z$ in the variable $s$,
 having a simple pole at $s=1$ with residue $1$ if and only if $z=1$.
Furthermore, for any positive integer $r$, by means of analytic continuation, we have
$$\Phi(z,\,1-r,\,w)=-\frac{\mathscr{B}_r(z,\,w)}{r}.$$
\end{theorem}

Theorem~\ref{thm:explicit} would be applicable to explicit evaluation of the Lerch zeta function
 at any specific point $(z,\,s,\,w) \in \complex \times \complex\setminus P_z \times \complex\setminus\rint_{\leq 0}$
 by taking sufficiently small $\epsilon>0$ and large $m,\,N$.
It also follows from Theorem~\ref{thm:anaconti} that by glueing the simply-connected sheets
 $\complex\setminus\cutz_{\cang'}$ and $\complex\setminus\cut{\cang}$ with
 $\cang \in \real\setminus2\pi\rint$, $\cang' \in \real$, $\Re(e^{i\cang'}) \geq 0$ together,
 we obtain the single-valued analytic continuation to the universal covering of the multiply-connected region
$\complex\setminus\{1\} \times \complex \times \complex\setminus\rint_{\leq 0}$
 with infinitely many punctures, which corresponds to the result~\cite[Theorem~3.6]{LL3}.
Note that the expression~\eqref{func:anaconti} enables us to evaluate the Lerch zeta function for any $z \in \real_{\geq 1}$,
 which has not been explicitly discussed in the literature.
Also, we obtain the formula for the special value $\Phi(e^{2\pi ia},\,1-r,\,w)$ with $a \in \real$
 for a positive integer $r$, which is originally due to Apostol~\cite{Apo}.
See Remark~\ref{thm:example} for a specific example of Theorem~\ref{thm:anaconti}.

Next we focus on functional equations for $\Phi(z,\,s,\,w)$ with the variable change $z=e^{2\pi ia}$.
In the paper~\cite{Ler}, Lerch proved the functional equation (called Lerch's transformation formula)
\begin{align}\label{eq:Lerch}
\Phi&(e^{2\pi ia},\,1-s,\,w)\\
&=\frac{\Gamma(s)}{(2\pi)^s}
\left\{e^{\pi i(\frac{s}{2}-2aw)}\Phi(e^{-2\pi iw},\,s,\,a)+e^{\pi i(-\frac{s}{2}+2(1-a)w)}\Phi(e^{2\pi iw},\,s,\,1-a)\right\} \notag
\end{align}
for $\Im(a)>0$ or $0<a<1$, $0<w<1$, and $\Re(s)>0$
 (actually proved for all $s \in \complex$ though it was not explicitly mentioned in~\cite{Ler})
using contour integral technic appeared as in Riemann's first proof of the functional equation for the Riemann zeta function.
On the other hand, Apostol~\cite{Apo} proved, according to the line of Riemann's second proof,
 there exists another symmetric functional equation
\begin{align}\label{eq:Apostol}
\Lambda(a,\,1-s,\,w)=2(2\pi)^{-s}\left(\cos\frac{\pi s}{2}\right)\Gamma(s)e^{-2\pi iaw}\Lambda(-w,\,s,\,a)
\end{align}
for $0<a<1$, $0<w<1$ (not explicitly written in~\cite{Apo}), and all $s \in \complex$, where
$$\Lambda(a,\,s,\,w):=\Phi(e^{2\pi ia},\,s,\,w)+e^{-2\pi ia}\Phi(e^{-2\pi ia},\,s,\,1-w).$$
This equation was also found by Weil~\cite[p.57]{EF}.
As a consequence, Apostol showed that
if $0<a<1$ and $0<w<1$ then Lerch's equation~(\ref{eq:Lerch}) is obtained from the equation~(\ref{eq:Apostol}).

In \S3, we show the following theorem using Theorem~\ref{thm:anaconti}.
Let
\begin{align*}
\cutsym{\cang}&:=\{w \in \complex \mid w=n+1+re^{i(\cang+\pi)} \text{ for $r \geq 0,\,n \in \rint_{\geq 0}$}\} \cup \cut{\cang}.
\end{align*}

\begin{theorem}\label{thm:equivalence}
Fix $\cang' \in 2\pi\rint$. Let $U$ be the maximal domain in $\complex\setminus(\cutsym{\cang} \cup (\rint+i\real))$
 containing the interval $(0,\,1)\,(\subset \real)$, and let
\begin{align}\label{domain:equiv}
U \times \complex \times U,
\tag{$D_{{\rm eq}}$}
\end{align}
 be a domain of $(a,\,s,\,w)$.
Then, Lerch's equation~\eqref{eq:Lerch} holds on \eqref{domain:equiv}
 if and only if Apostol's equation~\eqref{eq:Apostol} holds on \eqref{domain:equiv}.
\end{theorem}

In fact, Lerch's and Apostol's equations hold on \eqref{domain:equiv}, respectively (Proposition~\ref{prop:LA}),
 and hence these are essentially the same equation on \eqref{domain:equiv}.

\section{Analytic continuation of $\Phi$ and its evaluation}

In this section, we prove Theorem~\ref{thm:explicit}.

\subsection{Apostol's generating function}

Apostol~\cite{Apo} introduced the sequence of rational functions
 $\mathscr{B}_r(z,\,w)$ $(r \in \rint_{\geq 0})$ in $z$, $w$ by means of the generating function
\begin{align}\label{def:Apo}
\sum_{r=0}^{\infty} \frac{\mathscr{B}_r(z,\,w)}{r!}t^r=\frac{te^{tw}}{e^tz-1}.
\end{align}
The specialization $z=1$ together with the variable $w$ in the definition yields the definition of the Bernoulli polynomial in $w$,
 and for $r \geq 1$, there is the recurrence relation
\begin{align}\label{recurrence_Ber}
\mathscr{B}_r(1,\,w)=\frac{1}{r+1}\sum_{k=0}^{r-1} (-1)^{r-k}\binom{r+1}{k} \mathscr{B}_k(1,\,w) \left((w-1)^{r-k+1}-w^{r-k+1}\right).
\end{align}
If we regard both $z$ and $w$ as variables, then
$\mathscr{B}_0(z,\,w)=0$, $\mathscr{B}_1(z,\,w)=1/(z-1)$, and for $r \geq 2$, there is the recurrence relation
\begin{align}\label{recurrence_Apo}
\mathscr{B}_r(z,\,w)=\frac{1}{z-1}\sum_{k=0}^{r-1} (-1)^{r-k-1}\binom{r}{k} \mathscr{B}_k(z,\,w) \left(z(w-1)^{r-k}-w^{r-k}\right).
\end{align}
Thus $\mathscr{B}_r(z,\,w) \in \rat[1/(z-1),\,w]$, having pole only at $z=1$ for $r \geq 1$,
 and holomorphic for $(z,\,w) \in \complex\setminus\{1\} \times \complex$ as a function in two complex variables $z,\,w$.
In the rest of the paper,
if the value of $z$ has been specified, then we regard the definition~(\ref{def:Apo})
 as that under the specialization of $z$, in order to simplify notation in the case $z=1$.
In this sense, $\mathscr{B}_r(1,\,w)$ is the Bernoulli polynomial.

We will apply the following lemma to the proof of Theorem~\ref{thm:explicit} (Proposition~\ref{prop:J}).

\begin{lemma}\label{lem:Apo}
For any $r \in \rint_{\geq 0}$, we have
\begin{align*}
|\mathscr{B}_r(z,\,w)| \leq
\begin{cases}
2^r \cdot r!\,(|w|+1)^r & \text{ if $z=1$},\\
\dfrac{2^{r-1} \cdot r!}{|z-1|^r}(|z|+1)^{r-1}(|w|+1)^{r-1} & \text{ if $z \ne 1$}.
\end{cases}
\end{align*}
\end{lemma}

\begin{proof}
By induction. The case $r=0$ is clear from
$$\mathscr{B}_0(1,\,w)=1, \quad \mathscr{B}_0(z,\,w)=0 \quad (z \ne 1).$$
Consider the case $r \geq 1$.
Assume that the statement of the lemma holds for any $k \in \rint_{\geq 0}$ with $k<r$.
Then, by the recurrence relation~\eqref{recurrence_Ber} we have
\begin{align*}
|\mathscr{B}_r(1,\,w)|
&\leq \frac{1}{r+1}\sum_{k=0}^{r-1} \binom{r+1}{k} |\mathscr{B}_k(1,\,w)| \sum_{l=0}^{r-k} \binom{r-k+1}{l} |w|^l\\
&\leq \frac{1}{r+1}\sum_{k=0}^{r-1} \binom{r+1}{k} |\mathscr{B}_k(1,\,w)| (r-k+1)(|w|+1)^{r-k}\\
&\leq r!\,(|w|+1)^r\sum_{k=0}^{r-1} \frac{2^k}{(r-k)!}\\
&\leq r!\,(|w|+1)^r\sum_{k=0}^{r-1} 2^k\\
&\leq 2^r \cdot r!\,(|w|+1)^r.
\end{align*}
Also, using the recurrence relation~\eqref{recurrence_Apo} yields
\begin{align*}
|\mathscr{B}_r(z,\,w)|
&\leq \frac{1}{|z-1|}\sum_{k=0}^{r-1} \binom{r}{k} |\mathscr{B}_k(z,\,w)| (|z|+1)(|w|+1)^{r-k}\\
&\leq \frac{1}{|z-1|}\sum_{k=0}^{r-1} \binom{r}{k} \dfrac{2^{k-1} \cdot k!}{|z-1|^k} (|z|+1)^k(|w|+1)^{r-1}\\
&\leq \frac{1}{|z-1|^r}\sum_{k=0}^{r-1} \binom{r}{k} 2^{k-1} \cdot k!\,(|z|+1)^{r-1}(|w|+1)^{r-1}\\
&\leq \frac{r!}{|z-1|^r}(|z|+1)^{r-1}(|w|+1)^{r-1}\sum_{k=0}^{r-1} \frac{2^{k-1}}{(r-k)!}\\
&\leq \frac{2^{r-1} \cdot r!}{|z-1|^r}(|z|+1)^{r-1}(|w|+1)^{r-1}.
\end{align*}
\end{proof}

\subsection{Proof of Theorem~\ref{thm:explicit}}

Let $N$ be any non-negative integer, and let
\begin{align*}
D_z&:=\left\{z \in \complex \mid |z|<1\right\},\\
D_s&:=\left\{s \in \complex \mid \Re(s)>1\right\},\\
\wt{D}_w&:=\left\{w \in \complex \mid \Re(w)>-N\right\}.
\end{align*}
We start with six simply-connected domains
\begin{align}\label{basedomains}
\begin{array}{l}
\dom{1} D_z \times D_s \times \wt{D}_w\setminus\cut{\cang},\\
\dom{2} D_z \times D_s \text{ with fixed $w \in (\cut{\cang}\setminus\rint_{\leq 0}) \cap \wt{D}_w$},\\
\dom{1'} D_s \times \wt{D}_w\setminus\cut{\cang} \text{ with fixed $z \ne 1,\,|z|=1$},\\
\dom{3} D_s \times \wt{D}_w\setminus\cut{\cang} \text{ with $z=1$},\\
\dom{2'} D_s \text{ with fixed $z \ne 1,\,|z|=1$, $w \in (\cut{\cang}\setminus\rint_{\leq 0}) \cap \wt{D}_w$},\\
\dom{4} D_s \text{ with $z=1$ and fixed $w \in (\cut{\cang}\setminus\rint_{\leq 0}) \cap \wt{D}_w$}.
\end{array}\tag{$D$}
\end{align}
For each $j \in \{1,\,2,\,1',\,3,\,2',\,4\}$, since \eqref{basedomains}$(j) \subset$ \eqref{initdomains}$(j)$,
 the infinite series~\eqref{def:init}
 defines a single-valued function holomorphic on \eqref{basedomains}$(j)$
 in three variables for $j=1$, two variables for $j \in \{2,\,1',\,3\}$,
 and one variable for $j \in \{2',\,4\}$.
We regard the real interval $[a,\,b]\,(\subset \real)$ as the oriented line segment on $\real$ from $a$ to $b$,
 also $[a,\,\infty)$ as the oriented half line on $\real$ from $a$ to $\infty$, and so on.
On each domain of~(\ref{basedomains}) with restricted $w \in \real_{>0}$, consider the typical integral representation
\begin{align*}
\frac{z^n}{(n+w)^s}\Gamma(s)
&=e^{-4\pi i \nu s}\int_{[0,\,\infty)} e^{-t} z^n \left(\frac{t}{n+w}\right)^s \frac{dt}{t}, \quad n \geq N,
\end{align*}
where $\Gamma(s)$ denotes the gamma function and $(\frac{t}{n+w})^s=e^{s\log\frac{t}{n+w}}$.
Since $n+w \in \real_{>0}$, the substitution $t/(n+w) \to t$ yields the analytic continuation
\begin{align*}
\frac{z^n}{(n+w)^s}\Gamma(s)
&=e^{-4\pi i \nu s}\int_{[0,\,\infty)} e^{-t(n+w)} z^n t^{s-1} dt, \quad n \geq N,
\end{align*}
 where the right hand side of this equality is holomorphic for all
 $(z,\,s,\,w) \in \complex \times D_s \times \wt{D}_w$.
Multiplying $\Gamma(s)^{-1}$, taking the sum with respect to $n$, and using the condition of $D_z$, $D_s$, $\wt{D}_w$, we have
\begin{align}\label{eq:exp}
\Phi(z,\,s,\,w)
&=\sum_{n=0}^{N-1}\frac{z^n}{(n+w)^s}+\frac{1}{e^{4\pi i \nu s}\Gamma(s)}\int_{[0,\,\infty)} \left(\sum_{n=N}^{\infty} e^{-t(n+w)} z^n\right)t^{s-1} dt \notag\\
&=\sum_{n=0}^{N-1}\frac{z^n}{(n+w)^s}
+\frac{z^N}{e^{4\pi i \nu s}\Gamma(s)}\int_{[0,\,\infty)} \frac{e^{-t(N+w)}}{1-e^{-t}z}t^{s-1} dt \notag\\
&=\sum_{n=0}^{N-1}\frac{z^n}{(n+w)^s} +\frac{z^N}{e^{4\pi i \nu s}\Gamma(s)}(I_N+J_{N,\,m}+K_{N,\,m})
\end{align}
 on each domain of \eqref{basedomains}, where $m$ denotes any non-negative integer, and
\begin{align*}
I_N(z,\,s,\,w)&:=\int_{[\epsilon,\,\infty)} \phi_N(t)t^{s-2} dt,\\
J_{N,\,m}(z,\,s,\,w)&:=\int_{[0,\,\epsilon]} \left(\phi_N(t)-\phi_{N,\,m}(t)\right)t^{s-2} dt,\\
K_{N,\,m}(z,\,s,\,w)&:=\int_{[0,\,\epsilon]} \phi_{N,\,m}(t)t^{s-2} dt,\\
\phi_N(z,\,w,\,t)&:=\frac{te^{-t(N+w)}}{1-e^{-t}z} \quad (=:\phi_N(t)),\\
\phi_{N,\,m}(z,\,w,\,t)&:=\sum_{r=0}^m \frac{\mathscr{B}_r(z,\,N+w)}{r!}(-1)^r t^r \quad (=:\phi_{N,\,m}(t)).
\end{align*}
Here $\epsilon$ denotes any positive real, and $\phi_{N,\,m}(z,\,w,\,t)$ is a rational function
 in either $\rat[w,\,t]$ or $\rat[1/(z-1),\,w,\,t]$ according to $z=1$ or not
(Recall that $\mathscr{B}_r(1,\,N+w)$ stands for the Bernoulli polynomial, to ease the notation).

On the domains~\eqref{basedomains}, we have
\begin{align*}
K_{N,\,m}(z,\,s,\,w)
&=\int_{[0,\,\epsilon]} \sum_{r=0}^m \frac{\mathscr{B}_r(z,\,N+w)}{r!}(-1)^r t^{r+s-2} dt\\
&=\sum_{r=0}^m \frac{\mathscr{B}_r(z,\,N+w)}{r!}\frac{(-1)^r\epsilon^{s+r-1}}{s+r-1},
\end{align*}
 which is a meromorphic function in $\rat[\log\epsilon,\,1/(z-1),\,w,\,1/((s-1)s(s+1)\cdots(s+m-1))][\![s]\!]$ when $z \ne 1$,
 or $\rat[\log\epsilon,\,w,\,1/((s-1)s(s+1)\cdots(s+m-1))][\![s]\!]$ when $z=1$, having only simple poles,
 and hence the function $K_{N,\,m}/\Gamma(s)$ appearing in \eqref{eq:exp}
 is holomorphic for any $(z,\,s,\,w) \in \complex \setminus \{1\} \times \complex \times \complex$ except for $s=1$
 by the fact that the entire function $1/\Gamma(s)$ has simple zeros only at $s \in \rint_{\leq 0}$.
The (non-)holomorphy at $s=1$ follows from the fact that
 $\mathscr{B}_0(z,\,N+w)$ is equal to $0$ when $z \ne 1$, or $1$ when $z=1$,
 which leads us to the definition of $P_z$.

Let
\begin{align*}
\wt{D}_z&:=\left\{z \in \complex \mid z \notin [1,\,e^{\epsilon}]\right\},\\
\wt{D}_s&:=\left\{s \in \complex \mid \Re(s)>-m\right\}.
\end{align*}
Now we take a real $\epsilon>0$ satifying the condition~\eqref{ass1:epsilon}.
Then, by the results in \S\ref{subsec:hol}, \S\ref{subsec:AI}, \S\ref{subsec:MI},
 there exists a single-valued holomorphic function $\wt{I}_N$
 on the domains~\eqref{domains_wIm_single} giving analytic continuations of $I_N$ on each domain of \eqref{basedomains}.
Also, the function $J_{N,\,m}$ is single-valued and holomorphic on \eqref{domains_wIm_single}.

Summarize the discussion above.
For any $m,\,N \in \rint_{\geq 0}$, $\epsilon \in \real_{>0}$ under the condition~\eqref{ass1:epsilon}, we have the expression
\begin{align}\label{eq:anaconti}
\Phi(z,\,s,\,w)
=&\sum_{n=0}^{N-1}\frac{z^n}{(n+w)^s}
+\frac{z^N}{e^{4\pi i \nu s}\Gamma(s)}\sum_{r=0}^m \frac{\mathscr{B}_r(z,\,N+w)}{r!}\frac{(-1)^r\epsilon^{s+r-1}}{s+r-1}\\
&+\frac{z^N}{e^{4\pi i \nu s}\Gamma(s)}\wt{I}_N(z,\,s,\,w)+\frac{z^N}{e^{4\pi i \nu s}\Gamma(s)}J_{N,\,m}(z,\,s,\,w) \notag
\end{align}
 on the domains~(\ref{basedomains}), whose right hand side is both single-valued and holomorphic on four domains
\begin{align}
\begin{array}{l}
\dom{1} \wt{D}_z\setminus\cutz_{\cang'} \times \wt{D}_s \times \wt{D}_w\setminus\cut{\cang},\\
\dom{2} \wt{D}_z\setminus\cutz_{\cang'}\times \wt{D}_s \text{ with fixed $w \in (\cut{\cang}\setminus\rint_{\leq 0}) \cap \wt{D}_w$},\\
\dom{3} \wt{D}_s\setminus P_z \times \wt{D}_w\setminus\cut{\cang} \text{ with fixed $z \in \cutz_{\cang'} \setminus (1,\,e^{\epsilon}]$},\\
\dom{4} \wt{D}_s\setminus P_z \text{ with fixed $z \in \cutz_{\cang'} \setminus (1,\,e^{\epsilon}],\,w \in (\cut{\cang}\setminus\rint_{\leq 0}) \cap \wt{D}_w$}.
\end{array}\tag{$\wt{D}$}
\end{align}
Then
\begin{align*}
\rnums{1} & \eqref{basedomains}(j) \subset \eqref{domains}(j) \text{ for $j \in \{1,\,2,\,3,\,4\}$},\\
\rnums{2} & \eqref{basedomains}(j') \subset \eqref{domains}(j) \text{ for $j \in \{1,\,2\}$}.
\end{align*}
Since \eqref{basedomains}$(j) \subset$ \eqref{initdomains}$(j)$ for each $j \in \{1,\,2,\,1',\,3,\,2',\,4\}$,
 the right hand side of \eqref{eq:anaconti} gives analytic continuations of $\Phi$ ({\it i.e.} \eqref{def:init}) on \eqref{initdomains}
 to \eqref{domains} in accordance with \rnum{1}), \rnum{2}).
Here we need to separate the domains~\eqref{domains_wIm_single} (with $P_z$)
 into four domains~\eqref{domains} because the function $(n+w)^s=e^{s\log(n+w)}$ in $w$
 is discontinuous on the half line $-n+e^{i\cang}\real_{\geq 0}\,(\subset \cut{\cang})$ when $s \in \complex\setminus\rint$.
The evaluation of $\wt{I}_N$, $J_{N,\,m}$ follows from Proposition~\ref{prop:I}, \ref{prop:J} in \S\ref{subsec:evaluation}.
This proves Theorem~\ref{thm:explicit}.

\begin{remark}
It turns out that the multivaluedness of the Lerch zeta function $\Phi$ on the domain $\wt{D}_z \times \wt{D}_s \times \wt{D}_w$
 arises from $(n+w)^s=e^{s\log(n+w)}$ in the terms
$$\sum_{n=0}^{N-1}\frac{z^n}{(n+w)^s}$$
 of \eqref{eq:exp} and the (non-trivial multivalued) functions $\wt{I}_N$, $J_{N,\,m}$ (\S\ref{subsec:MI}).
\end{remark}

\subsection{Holomorphy of $J_{N,\,m}(z,\,s,\,w)$}\label{subsec:hol}

Consider the function
\begin{align*}
J_{N,\,m}(z,\,s,\,w)=\int_{[0,\,\epsilon]} \left(\phi_N(t)-\phi_{N,\,m}(t)\right)t^{s-2} dt,
\end{align*}
where $t^{s-2}=e^{(s-2)\log t}$ with $\cang \leq \arg t < \cang+2\pi$ and $\cang \in \real\setminus2\pi\rint$.
Especially $\log t$ is holomorphic on the path $[0,\,\epsilon]$.

Since the function $\phi_N(t)$ is holomorphic
 for $(z,\,w,\,t)$ in any compact subset of $\wt{D}_z \times \wt{D}_w \times [0,\,\epsilon]$
 (resp. for all $(w,\,t) \in \wt{D}_w \times [0,\,\epsilon]$ with $z=1$),
 (Note that $\phi_N$ in two variables $z,\,t$ is discontinuous on those $(z,\,t)$ with $z=e^t$,
 in particular at $(z,\,t)=(e^t,\,t)$ for any $t \in [0,\,\epsilon]$),
 by the definition of $\mathscr{B}_r(z,\,N+w)$ we can expand $\phi_N$ to the Taylor series
$$\phi_N(t)=\sum_{r=0}^{\infty} \frac{\mathscr{B}_r(z,\,N+w)}{r!}(-1)^r t^r$$
around $t=0$, and hence $|\phi_N(t)-\phi_{N,\,m}(t)|=O(t^{m+1})$ as $t \to 0$
 on any compact subset $T \subset \wt{D}_z \times \wt{D}_s \times \wt{D}_w$
 (resp. $T \subset \wt{D}_s \times \wt{D}_w$ with $z=1$).
This implies that the absolute value of the integrand of $J_{N,\,m}$
 is equal to $O(t^{m+\Re(s)-1})$ as $t \to 0$ on the set $T$.
Thus the integral $J_{N,\,m}$ converges and is holomorphic
 for all $(z,\,s,\,w) \in \wt{D}_z \times \wt{D}_s \times \wt{D}_w$
 (resp. for all $(s,\,w) \in \wt{D}_s \times \wt{D}_w$ with $z=1$).

\subsection{Analytic continuation of $I_N(z,\,s,\,w)$ to $\complex \setminus \{e^{\epsilon}\} \times \complex \times \complex$}\label{subsec:AI}

Let $0<\epsilon<\alpha$.
For any oriented path $L \subset \complex$ from $\epsilon$ to $\alpha$, let
\begin{align*}
I_N^L(z,\,s,\,w):=\int_{L} \phi_N(t)t^{s-2} dt,
\end{align*}
where $t^{s-2}=e^{(s-2)\log t}$ with $\cang \leq \arg t < \cang+2\pi$ and $\cang \in \real\setminus2\pi\rint$.
One has $I_N^{[\epsilon,\,\alpha]} \to I_N$ as $\alpha \to \infty$.
The function $\log t$ is holomorphic for any $t \in \complex\setminus e^{i\cang}\real_{\geq 0}$.
Assume
\begin{align*}
L \subset H_t:=\complex\setminus e^{i\cang}\real_{\geq 0}.
\end{align*}
Since the function $\phi_N(t)$ is holomorphic
 for $(z,\,w,\,t)$ in any compact subset of $\complex \setminus e^L \times \complex \times L$
 (resp. for all $(w,\,t) \in \complex \times L$ with $z=1$), where $e^L=\{e^t \mid t \in L\}$
 (Note that $\phi_N$ in two variables $z,\,t$ is discontinuous on those $(z,\,t)$ with $z=e^t$,
 in particular at $(z,\,t)=(e^t,\,t)$ for any $t \in L$),
the integral $I_N^L$ converges and is holomorphic
 for all $(z,\,s,\,w) \in \complex \setminus e^L \times \complex \times \complex$
 (resp. for all $(s,\,w) \in \complex \times \complex$ with $z=1$).

In the case $\alpha=\infty$,
if $L$ satisfies the condition
\begin{align*}
\bullet\, &\inf\left\{\Re(t) \in \real \mid t \in L\right\}>-\infty\\
\bullet\, &^{\exists}\gamma_0 \in \real \text{ such that $\Im(t)$ is constant for any $t \in L$ with $\Re(t)>\gamma_0$}
\end{align*}
 then the integral $I_N^L$ still converges
 because for any compact subset $T$ of $\{(z,\,s,\,w) \in \complex \setminus e^L \times \complex \times \complex\}$
 there exists a positive constant $C \in \real$ depending on $N,\,L,\,T$ such that
$$|\phi_N(t)t^{s-2}|
=\frac{e^{-\Re(t)(N+\Re(w))+\Im(t)\Im(w)}|t|^{\Re(s)+1}e^{-\Im(s)\arg\,t}}{|1-e^{-\Re(t)-i\Im(t)}z|}|t|^{-2} \leq \frac{C}{|t|^2} \quad \text{ for $t \in L$}.$$
Especially, $I_N^{[\epsilon,\,\infty)}=I_N$ converges on
 the domain $\complex \setminus [e^{\epsilon},\,\infty) \times \complex \times \complex$
 containing any domain of \eqref{basedomains}.

Now we take a real $\epsilon>0$ satisfying
\begin{align}
\Re(e^{i\cang})<\Re(e^{i\arg(1+i\epsilon)})\,(=1/\sqrt{1+\epsilon^2}), \tag{$\epsilon$\,1}
\end{align}
 which is possible by the assumption $\cang \in \real\setminus2\pi\rint$ implying $\Re(e^{i\cang})<1$.
Then we show that the function
\begin{align*}
\wt{I}_N(z,\,s,\,w):=
\begin{cases}
I_N(z,\,s,\,w) & \text{ on $\complex \setminus [e^{\epsilon},\,\infty) \times \complex \times \complex$}\\
I_N^{L_{\epsilon}^{\infty}}(z,\,s,\,w) & \text{ on $(e^{\epsilon},\,\infty) \times \complex \times \complex$}
\end{cases}
\end{align*}
is holomorphic on the domain $\complex \setminus \{e^{\epsilon}\} \times \complex \times \complex$,
 where $L_{\epsilon}^{\alpha}\,(\subset H_t)$ denotes the oriented line segment in the complex $t$-plane connecting the four complex points
 $\epsilon$, $\epsilon(1+i\epsilon)$, $\alpha+i\epsilon^2$, $\alpha$, in this order (See Figure~2).
By the discussion above, $I_N\,(=I_N^{[\epsilon,\,\infty)})$ (resp. $I_N^{L_{\epsilon}^{\infty}}$) converges and is holomorphic on
 $\complex \setminus [e^{\epsilon},\,\infty) \times \complex \times \complex$
 (resp. $\complex \setminus e^{L_{\epsilon}^{\infty}} \times \complex \times \complex$),
 where $(e^{\epsilon},\,\infty) \subset \complex \setminus e^{L_{\epsilon}^{\infty}}$.
\begin{figure}[ht]
\centering
\includegraphics{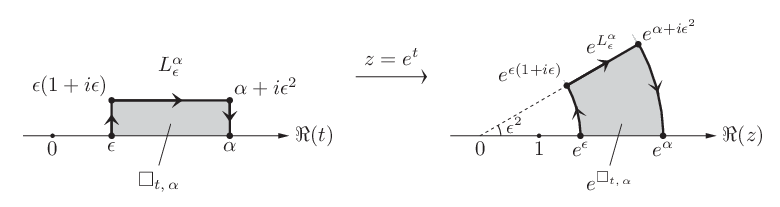}
\caption{\small the paths $L_{\epsilon}^{\alpha},\,e^{L_{\epsilon}^{\alpha}}$ and the regions $\square_{t,\,\alpha},\,e^{\square_{t,\,\alpha}}$}
\end{figure}
Let
\begin{align*}
\square_{t,\,\alpha}&:=\left\{t \in \complex \midflex \epsilon \leq \Re(t) \leq \alpha,\,0 \leq \Im(t) \leq \epsilon^2\right\},\\
R_z&:=\left\{re^{i\theta} \in \complex \midflex r>0,\,\frac{\pi}{2}<\theta<\frac{3\pi}{2}\right\}.
\end{align*}
Then $\square_{t,\,\alpha} \subset H_t$ (by the condition~\eqref{ass1:epsilon}),
 $e^{\square_{t,\,\alpha}} \cap \closure{R_z}=\emptyset$, where $\closure{R_z}$ denotes the closure of $R_z$,
 and hence the integrand of $I_N^{L_{\epsilon}^{\alpha}}$ is holomorphic for any $(z,\,s,\,w,\,t) \in R_z \times \complex \times \complex \times \square_{t,\,\alpha}$.
Thus, by Cauchy-Goursat theorem,
 we have $I_N^{L_{\epsilon}^{\alpha}}=I_N^{[\epsilon,\,\alpha]}$ (and hence $I_N^{L_{\epsilon}^{\infty}}=I_N$)
 on the domain $R_z \times \complex \times \complex$.
Therefore, since $R_z \subset (\complex \setminus [e^{\epsilon},\,\infty)) \cap (\complex \setminus e^{L_{\epsilon}^{\infty}})$,
 $\wt{I}_N$ is holomorphic on $\complex \setminus \{e^{\epsilon}\} \times \complex \times \complex$ by analytic continuation.

\subsection{Multivaluedness of $\wt{I}_N,\,J_{N,\,m}$ on $\wt{D}_z \times \wt{D}_s \times \wt{D}_w$}\label{subsec:MI}

When a complex $z \in \complex$ moves continuously on the multiply-connected domain $\complex \setminus \{1\}$ and returns to original position,
 the oriented line segment $e^{-t}z\,(t \in [0,\,\infty))$ on $\complex$ could pass through the point $1 \in \real$
 and returns to the original position.
This implies that a non-trivial multivaluedness of the functions $\wt{I}_N$, $J_{N,\,m}$ on the domain
 $\wt{D}_z \times \wt{D}_s \times \wt{D}_w$ in the variable $z$ occurs in a loop around the point $1$,
 because the function $\phi_N$ has simple poles at $t \in \complex$ with $e^t=z$ for any fixed $z$ in $\complex \setminus \{1\}$.
Therefore, in order to make both $\wt{I}_N$ and $J_{N,\,m}$ single-valued for $z \in \wt{D}_z$, we need a cut
$$\cutz_{\cang'}=\{z \in \complex \mid z-1=re^{i\cang'} \text{ for $r \geq 0$}\} \quad \text{ for some $\cang' \in \real$}$$
 in the $z$-plane (See Figure~3).
\begin{figure}[ht]
\centering
\includegraphics{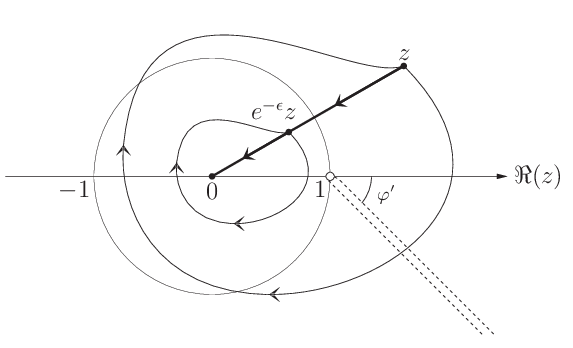}
\caption{\small A movement of the line segment $e^{-t}z\,(t \in [0,\,\epsilon] \cup [\epsilon,\,\infty))$}
\end{figure}
\noindent
Since the infinite series~(\ref{def:init}) is no longer holomorphic on the initial domain $D_z$ when $\Re(e^{i\cang'})<0$,
 the condition $\Re(e^{i\cang'}) \geq 0$ is necessary.
Thus, we obtain single-valued holomorphic functions $\wt{I}_N$, $J_{N,\,m}$ on two domains
\begin{align}\label{domains_wIm_single}
\begin{array}{l}
\dom{1} \wt{D}_z \setminus \cutz_{\cang'} \times \wt{D}_s \times \wt{D}_w,\\
\dom{3} \wt{D}_s \times \wt{D}_w \text{ with fixed $z \in \cutz_{\cang'} \setminus (1,\,e^{\epsilon}]$ $(=(\cutz_{\cang'} \cap \wt{D}_z) \cup \{1\})$},
\end{array}
\end{align}
 giving single-valued analytic continuation of $I_N$, $J_{N,\,m}$ on the domains~\eqref{basedomains}
 in accordance with
\begin{align*}
\rnums{1} & \eqref{basedomains}(j) \subset \eqref{domains_wIm_single}(1) \text{ for $j \in \{1,\,2,\,1',\,2'\}$},\\
\rnums{2} & \eqref{basedomains}(j) \subset \eqref{domains_wIm_single}(3) \text{ for $j \in \{3,\,4\}$},\\
\rnums{3} & \wt{I}_N=I_N \text{ on each domain of \eqref{basedomains}}.
\end{align*}

\subsection{Evaluation of $\wt{I}_N$, $J_{N,\,m}$}\label{subsec:evaluation}

Let
\begin{align}\label{def:ABC}
A_{\epsilon}(z)&:=
\begin{cases}
1 & \text{if $\Re(z) \leq 0$},\\
\left|1-\Re(z)\dfrac{z}{|z|^2}\right| & \text{if $0<\frac{\Re(z)}{|z|^2} \leq e^{-\epsilon}$},\\
\left|1-e^{-\epsilon}z\right| & \text{if $e^{-\epsilon}<\frac{\Re(z)}{|z|^2}$},
\end{cases} \notag\\
B_{\epsilon}(z)&:=\min\left\{e^{-\epsilon}z-1,\,|\sin\epsilon^2|\right\},\\
C_{\epsilon}(s,\,w)&:=\sqrt{2}^{\Re(s)+1}\exp(\epsilon^2|\Im(w)|+|\Im(s)|(\tan^{-1}\!\epsilon+2\pi|\nu|)), \notag\\
E_{\epsilon,\,N}(s,\,w)&:=
\begin{cases}
\left(\dfrac{\Re(s)+1}{e(\Re(w)+N)}\right)^{\Re(s)+1} & \text{if $\frac{\Re(s)+1}{\Re(w)+N}>\epsilon$},\\
\dfrac{\epsilon^{\Re(s)+1}}{e^{\epsilon(\Re(w)+N)}} & \text{if $\frac{\Re(s)+1}{\Re(w)+N} \leq \epsilon$}.
\end{cases} \notag
\end{align}

\begin{lemma}\label{lem:lowerbound}
For any $z \in \complex \setminus \{e^{\epsilon}\}$,
\begin{align*}
|1-e^{-t}z| \geq
\begin{cases}
A_{\epsilon}(z)>0 & \text{if $t \in [\epsilon,\,\infty)$, $z \in \complex \setminus [e^{\epsilon},\,\infty)$},\\
B_{\epsilon}(z) & \text{if $t \in L_{\epsilon}^{\infty}$, $z \in (e^{\epsilon},\,\infty)$}.
\end{cases}
\end{align*}
Under the condition $(0<)\,\epsilon<\sqrt{\pi}$, we have $B_{\epsilon}(z)>0$.
\end{lemma}

\begin{proof}
In the case $t \in [\epsilon,\,\infty)$, $z \in \complex \setminus [e^{\epsilon},\,\infty)$,
 differentiating the real function
$$|1-e^{-t}z|^2=(e^{-t}\Re(z)-1)^2+(e^{-t}\Im(z))^2$$
in the variable $t$ yields $|1-e^{-t}z| \geq A_{\epsilon}(z)>0$.

Next assume that $t \in L_{\epsilon}^{\infty}$, $z \in (e^{\epsilon},\,\infty)$.
If $t=\epsilon+i\delta\,(0 \leq \delta \leq \epsilon^2)$ then
$$|1-e^{-t}z| \geq ||e^{-t}z|-1|=|e^{-\Re(t)}z-1|=e^{-\epsilon}z-1>0.$$
Otherwise, $t=\gamma+i\epsilon^2\,(\gamma \geq \epsilon)$, then differentiating the real function
$$|1-e^{-t}z|^2=(e^{-\gamma}z-1)^2+2e^{-\gamma}z(1-\cos\epsilon^2).$$
 in the variable $\gamma$ yields
$$(e^{-\gamma}z-1)^2+2e^{-\gamma}z(1-\cos\epsilon^2) \geq \sin^2\epsilon^2 \quad \text{ for $\gamma \geq \epsilon$},$$
and hence $|1-e^{-t}z| \geq |\sin\epsilon^2|$ ($>0$ if $0<\epsilon<\sqrt{\pi}$).
This proves $|1-e^{-t}z| \geq B_{\epsilon}(z)$.
\end{proof}

\begin{proposition}\label{prop:I}
Let $\epsilon$ be a positive real satisfying
\begin{align}
\Re(e^{i\cang})<\Re(e^{i\arg(1+i\epsilon)})\,(=1/\sqrt{1+\epsilon^2}). \tag{$\epsilon$\,1}
\end{align}
For any $(z,\,s,\,w) \in \complex \setminus \{e^{\epsilon}\} \times \complex \times \wt{D}_w$, we have
\begin{align*}
|\wt{I}_N(z,\,s,\,w)| &\leq
\begin{cases}
\dfrac{1}{\epsilon A_{\epsilon}(z)}E_{\epsilon,\,N}(s,\,w) & \text{if $z \in \complex \setminus [e^{\epsilon},\,\infty)$},\\
\left(1+\frac{1}{\epsilon}\right)\dfrac{C_{\epsilon}(s,\,w)}{B_{\epsilon}(z)}E_{\epsilon,\,N}(s,\,w) & \text{if $z \in (e^{\epsilon},\,\infty)$}.\\
\end{cases}
\end{align*}
Especially, $\wt{I}_N \to 0$ as $N \to \infty$ on any compact subset of $\complex \setminus \{e^{\epsilon}\} \times \complex \times \wt{D}_w$.
\end{proposition}

\begin{proof}
In the case $z \in \complex \setminus [e^{\epsilon},\,\infty)$, $\wt{I}_N=I_N$, and
 differentiating the real function $e^{-t(N+\Re(w))}|t^{s+1}|$ in the variable $t\,(\geq \epsilon)$ yields
\begin{align*}
e^{-t(N+\Re(w))}|t^{s+1}| \leq E_{\epsilon,\,N}(s,\,w).
\end{align*}
By Lemma~\ref{lem:lowerbound}, we have the evaluation
\begin{align*}
|I_N| &\leq \int_{\epsilon}^{\infty} \frac{e^{-t(N+\Re(w))}|t^{s+1}|}{|1-e^{-t}z|}\frac{dt}{t^2}\\
&\leq \frac{1}{\epsilon A_{\epsilon}(z)}E_{\epsilon,\,N}(s,\,w).
\end{align*}

Consider the case $z \in (e^{\epsilon},\,\infty)$.
For any $t \in L_{\epsilon}^{\infty}$,
 it follows from the assumption~\eqref{ass1:epsilon} that $0 \leq \arg\,t-2\pi\nu \leq \tan^{-1}\!\epsilon$,
 and one can verify
\begin{align*}
|t|<\sqrt{2}\,\Re(t), \quad e^{\Im(t)\Im(w)} \leq e^{\epsilon^2|\Im(w)|}, \quad |t^{s+1}| \leq |t|^{\Re(s)+1}e^{|\Im(s)|(\tan^{-1}\!\epsilon+2\pi|\nu|)},
\end{align*}
and hence
\begin{align*}
\left|t^{s+1}e^{-t(N+w)}\right|
&\leq |t^{s+1}|e^{-\Re(t)(\Re(w)+N)}e^{\Im(t)\Im(w)}\\
&\leq \sqrt{2}^{\Re(s)+1}e^{\epsilon^2|\Im(w)|}e^{|\Im(s)|(\tan^{-1}\!\epsilon+2\pi|\nu|)} \cdot \Re(t)^{\Re(s)+1}e^{-\Re(t)(\Re(w)+N)}\\
&\leq C_{\epsilon}(s,\,w)E_{\epsilon,\,N}(s,\,w).
\end{align*}
Here the last evaluation for the case $\Re(s)+1>0$ follows from the differential of the real function $x^{\Re(s)+1}e^{-x(\Re(w)+N)}$ in the variable $x\,(=\Re(t) \geq \epsilon)$.
Using Lemma~\ref{lem:lowerbound} yields
\begin{align*}
|\wt{I}_N| &\leq \int_{L_{\epsilon}^{\infty}} \frac{|t^{s+1}e^{-t(N+w)}|}{|1-e^{-t}z|} \frac{|dt|}{|t^2|}\\
&\leq \frac{C_{\epsilon}(s,\,w)}{B_{\epsilon}(z)}E_{\epsilon,\,N}(s,\,w) \int_{L_{\epsilon}^{\infty}}\frac{|dt|}{|t|^2}.
\end{align*}
The statement of the proposition follows from
\begin{align*}
\int_{L_{\epsilon}^{\infty}}\frac{|dt|}{|t|^2}
&=\int_0^{\epsilon^2} \frac{d\delta}{\epsilon^2+\delta^2}+\int_{\epsilon}^{\infty} \frac{d\gamma}{\gamma^2+\epsilon^4} \quad (t=\gamma+i\delta)\\
&\leq \int_0^{\epsilon^2} \frac{d\delta}{\epsilon^2}+\int_{\epsilon}^{\infty} \frac{d\gamma}{\gamma^2}\\
&=1+\frac{1}{\epsilon}.
\end{align*}
\end{proof}

\begin{proposition}\label{prop:J}
Let $N,\,m \in \rint_{\geq 0}$, $(z,\,s,\,w) \in \complex \times \wt{D}_s \times \complex$,
 and let $\epsilon$ be a real satisfying
\begin{align}\label{ass2:epsilon}
0 < \epsilon \leq
\begin{cases}
\dfrac{1}{4(|N+w|+1)} & \text{ if $z=1$},\\
\dfrac{|z-1|}{4(|z|+1)(|N+w|+1)} & \text{ if $z \ne 1$}.
\end{cases} \tag{$\epsilon$\,2}
\end{align}
Then,
$$\left|J_{N,\,m}(z,\,s,\,w)\right| \leq
\begin{cases}
\dfrac{1}{2^m}\dfrac{\epsilon^{\Re(s)-1}}{\Re(s)+m} & \text{ if $z=1$},\\
\dfrac{2^{-m+1}}{|z-1|}\dfrac{\epsilon^{\Re(s)}}{\Re(s)+m} & \text{ if $z \ne 1$}.
\end{cases}
$$
Especially, $J_{N,\,m} \to 0$ as $m \to \infty$ on any compact subset of $\complex \times \wt{D}_s \times \complex$.
\end{proposition}

\begin{proof}
We show the statement by Lemma~\ref{lem:Apo}.
In the case $z \ne 1$, let
$$F(z,\,w):=\dfrac{2}{|z-1|}(|z|+1)(|w|+1) \quad (\geq 2).$$
Then, for any $t \in \complex$ with $|t|<F(z,\,N+w)^{-1}$, we have
\begin{align*}
|\phi_N(t)-\phi_{N,\,m}(t)|
&\leq \sum_{r=m+1}^{\infty} \frac{|\mathscr{B}_r(z,\,N+w)|}{r!}|t|^r\\
&\leq \sum_{r=m+1}^{\infty} \frac{F(z,\,N+w)^{r-1}}{|z-1|}|t|^r\\
&=\frac{F(z,\,N+w)^m|t|^{m+1}}{|z-1|\left(1-F(z,\,N+w)|t|\right)},
\end{align*}
and hence
\begin{align*}
\left|\int_0^{\epsilon} \left(\phi_N(t)-\phi_{N,\,m}(t)\right)t^{s-2} dt\right|
&\leq \int_0^{\epsilon} \frac{F(z,\,N+w)^m|t|^{m+1}}{|z-1|\left(1-F(z,\,N+w)|t|\right)} \cdot |t^{s-2}| dt\\
&\leq \frac{F(z,\,N+w)^m}{|z-1|\left(1-F(z,\,N+w)\epsilon\right)} \int_0^{\epsilon} t^{\Re(s)+m-1} dt\\
&\leq \frac{2F(z,\,N+w)^m}{|z-1|}\frac{\epsilon^{\Re(s)+m}}{\Re(s)+m} \quad (\to 0 \text{ as $\epsilon \to 0$})\\
&\leq \frac{2^{-m+1}}{|z-1|}\frac{\epsilon^{\Re(s)}}{\Re(s)+m} \quad (\to 0 \text{ as $m \to \infty$}).
\end{align*}
Similarly, one can prove the case $z=1$.
\end{proof}

\subsection{The special values of $\Phi(z,\,s,\,w)$ at non-positive integers in $s$}\label{subsec:value}

For any $r \in \rint_{>0}$,
 taking an integer larger than or equal to $r$ as the integer $m$ in the expression~(\ref{func:anaconti})
 yields the special value
\begin{align*}
\Phi(z,\,1-r,\,w)
&=\sum_{n=0}^{N-1}\frac{z^n}{(n+w)^{1-r}}+\lim_{s \to 1-r} \frac{z^N}{\Gamma(s)}
\frac{\mathscr{B}_r(z,\,N+w)}{r!}\frac{(-1)^r\alpha^{s+r-1}}{s+r-1}\\
&=\sum_{n=0}^{N-1}\frac{z^n}{(n+w)^{1-r}}-\frac{\mathscr{B}_r(z,\,N+w)z^N}{r}\\
&=-\frac{\mathscr{B}_r(z,\,w)}{r}.
\end{align*}
For the first and second equalities we use the fact that $\Gamma(s)$ has simple poles only at $s=1-r$ with residue $(-1)^{1-r}/(r-1)!$.
The last equality follows from Lemma~\ref{lem:B}.

\begin{lemma}\label{lem:B}
For any positive integer $r$ and any non-negative integer $N$,
$$z \cdot \mathscr{B}_r(z,\,N+1+w)-\mathscr{B}_r(z,\,N+w)=r(N+w)^{r-1}.$$
Equivalently,
$$\frac{\mathscr{B}_r(z,\,N+w)z^N}{r}-\frac{\mathscr{B}_r(z,\,w)}{r}=\sum_{n=0}^{N-1} \frac{z^n}{(n+w)^{1-r}}.$$
\end{lemma}

\begin{proof}
It is easily seen from definition that
\begin{align*}
z \cdot \sum_{r=0}^{\infty} \frac{\mathscr{B}_r(z,\,N+1+w)}{r!}t^r-&\sum_{r=0}^{\infty} \frac{\mathscr{B}_r(z,\,N+w)}{r!}t^r\\
&=z\frac{te^{t(N+1+w)}}{e^tz-1}-\frac{te^{t(N+w)}}{e^tz-1}\\
&=te^{t(N+w)}\\
&=\sum_{r=1}^{\infty} \frac{r(N+w)^{r-1}}{r!}t^r.
\end{align*}
\end{proof}

\begin{remark}
The specialization $w=0$ in Lemma~\ref{lem:B} yields the result~\cite{Apo2}.
\end{remark}

\subsection{Corollaries and Remarks of Theorem~\ref{thm:explicit}}\label{subsec:cor}

By the specialization $z=1$ or $w=1$ in Theorem~\ref{thm:anaconti},
 we have results on the Hurwitz zeta function or the polylogarithm, respectively.

\begin{corollary}
The Hurwitz zeta function $\zeta(s,\,w)=\Phi(1,\,s,\,w)$ has single-valued analytic continuations holomorphic on two domains
\begin{align*}
\begin{array}{l}
\dom{1} \complex\setminus\{1\} \times \complex\setminus\cut{\varphi} \text{ in two variables $(s,\,w)$},\\
\dom{2} \complex\setminus\{1\} \text{ in one variable $s$ with fixed $w \in \cut{\varphi}\setminus\rint_{\leq 0}$}.
\end{array}
\end{align*}
This function has a simple pole at $s=1$ with residue $1$.
\end{corollary}

\begin{corollary}
The polylogarithm $\plog_s(z)=z\Phi(z,\,s,\,1)$ has single-valued analytic continuations holomorphic on two domains
\begin{align*}
\begin{array}{l}
\dom{1} \complex\setminus\cutz_{\cang'} \times \complex \text{ in two variables $(z,\,s)$},\\
\dom{2} \complex\setminus\{1\} \text{ in one variable $s$ with fixed $z \in \cutz_{\cang'}$}.
\end{array}
\end{align*}
In particular, if $z=1$ then the second case corresponds to the Riemann zeta function.
\end{corollary}

\begin{remark}\label{thm:example}
We give an explicit example of $\plog_1(z)=z\Phi(z,\,1,\,1)$.
Taking $N=0,\,m=1$ for the equation~\eqref{func:anaconti} yields the analytic continuation
\begin{align*}
\Phi(z,\,1,\,1)=-\epsilon\cdot\mathscr{B}_1(z,\,1)+\wt{I}_0(z,\,1,\,1)+J_{0,\,1}(z,\,1,\,1),
\end{align*}
 where the right hand side is both single-valued and holomorphic on the domain~\eqref{domains}(1) with fixed $s=w=1$.
In the case $z \in \wt{D}_z\setminus(\cutz_{\cang'} \cup [e^{\epsilon},\,\infty))$ we have
 $\wt{I}_0=I_0$, and hence
\begin{align*}
z\Phi(z,\,1,\,1)
&=z\int_{[0,\,\infty)} \frac{e^{-t}}{1-e^{-t}z} dt\\
&=z\int_{[0,\,\infty)} \frac{-e^{-t}e^{i(\cang-\cang')}}{(e^{-t}z-1)e^{i(\cang-\cang')}} dt\\
&=\left[\log((e^{-t}z-1)e^{i(\cang-\cang')})\right]_0^{\infty}\\
&=i\arg(-e^{i(\cang-\cang')})-\log((z-1)e^{i(\cang-\cang')}),
\end{align*}
 which is both single-valued and holomorphic on $\complex\setminus\cutz_{\cang'}$
 with principal value $\cang \leq \arg \lambda < \cang+2\pi$ for $\lambda \in \complex^*$.
In particular, the substitution $\cang=-\pi$, $\cang'=0$ yields a usual single-valued function
 $z\Phi(z,\,1,\,1)=-\log(1-z)$, holomorphic on $\complex\setminus[1,\,\infty)$ as in~\cite{HTF}, \cite{KKY}.
\end{remark}

\begin{remark}
In general, by the condition $\cang \leq \arg \lambda < \cang+2\pi$ for $\lambda \in \complex^*$,
 the function $z \mapsto \arg((z-1)e^{i(\cang-\cang')})$ is both single-valued and holomorphic on $\complex\setminus\cutz_{\cang'}$.
Especially, we can take $\cang'$ as $\cang'=\cang$ when $\Re(e^{i\cang}) \geq 0$, or $\cang'=\cang+\pi$ when $\Re(e^{i\cang})<0$,
 which means
$$
\cutz_{\cang'}=
\begin{cases}
\{z \in \complex \mid z-1=re^{i\cang} \text{ for $r \geq 0$}\} & \text{ if $\Re(e^{i\cang}) \geq 0$}\\
\{z \in \complex \mid 1-z=re^{i\cang} \text{ for $r \geq 0$}\} & \text{ if $\Re(e^{i\cang}) < 0$}.
\end{cases}
$$
\end{remark}

\section{Equivalence of Lerch's equation and Apostol's equation}

In this section, on a large region of $(a,\,s,\,w)$,
 we show that Lerch's equation~(\ref{eq:Lerch}) holds if and only if Apostol's equation~(\ref{eq:Apostol}) holds
 (Theorem~\ref{thm:equivalence}), meaning that these two equations are essentially the same.
We first extend both domains of Lerch's equation and Apostol's equation using Theorem~\ref{thm:anaconti}.
For simplicity, let $\cang' \in 2\pi\rint$.
Then $\cutz_{\cang'}=[1,\,\infty)$.

\subsection{Lerch's and Apostol's equations on extended domains}\label{subsec:LA}

If $\Phi(z,\,s,\,w)$ is holomorphic at $z=e^{2\pi i\gamma}\,(\gamma \in \complex)$
 then $\Phi(e^{2\pi ia},\,s,\,w)$ is holomorphic at all $a \in \gamma+\rint$
 because the complex exponential function $z=e^{2\pi ia}$ is an entire function,
 and thus we can apply Theorem~\ref{thm:anaconti} to $\Phi(e^{2\pi ia},\,s,\,w)$.

\begin{proposition}\label{prop:LA}
The left and right hand side of Lerch's equation~\eqref{eq:Lerch} are holomorphic on two open sets
\begin{align*}
\begin{array}{l}
\dom{1} (a,\,s,\,w) \in \complex\setminus(\cutsym{\cang} \cup (\rint+i\real_{\leq 0})) \times \complex \times \complex\setminus(\cut{\cang} \cup (\rint+i\real)),\\
\dom{2} s \in \complex \text{ with fixed $a \in (\rint+i\real_{\leq 0})\setminus\cutsym{\cang}$, $w \in \rint_{>0}$}.
\end{array}
\end{align*}
In particular, Lerch's equation~\eqref{eq:Lerch} holds on the domain
\begin{align}\label{domain:Lerch}
(a,\,s,\,w) \in U_a \times \complex \times U_w,
\tag{$D_{\mathrm{L}}$}
\end{align}
 where $U_a$ denotes a domain in $\complex\setminus(\cutsym{\cang} \cup (\rint+i\real_{\leq 0}))$
 satisfying one of two conditions
\begin{align*}
\rnums{1} & \text{$U_a$ contains the interval $(0,\,1)\,(\subset \real)$},\\
\rnums{2} & U_a \cap \{a \in \complex \mid \Im(a)>0\} \ne \emptyset,
\end{align*}
 and $U_w$ the maximal domain in $\complex\setminus(\cut{\cang} \cup (\rint+i\real))$ containing $(0,\,1)$.
Also, the left and right hand side of Apostol's equation~\eqref{eq:Apostol} are holomorphic on the open set
\begin{align*}
(a,\,s,\,w) \in \complex\setminus(\cutsym{\cang} \cup (\rint+i\real)) \times \complex \times \complex\setminus(\cutsym{\cang} \cup (\rint+i\real)).
\end{align*}
In particular, Apostol's equation~\eqref{eq:Apostol} holds on the domain
\begin{align}\label{domain:Apostol}
(a,\,s,\,w) \in U \times \complex \times U,
\tag{$D_{\mathrm{A}}$}
\end{align}
 where $U$ denotes the maximal domain in $\complex\setminus(\cutsym{\cang} \cup (\rint+i\real))$ containing $(0,\,1)$.
Especially, both Lerch's and Apostol's equations hold on the domain~\eqref{domain:Apostol} $(=$ \eqref{domain:Lerch} by taking $U_a=U_w=U)$.
\end{proposition}

\begin{proof}
For Lerch's equation, it suffices to find the intersection of domains of
 $\Phi(e^{2\pi ia},\,1-s,\,w)$, $\Phi(e^{-2\pi iw},\,s,\,a)$, and $\Phi(e^{2\pi iw},\,s,\,1-a)$
 on each domain of \eqref{finaldomains} in Theorem~\ref{thm:anaconti}.

\noindent
\eqref{finaldomains}(1) On the domain $\complex\setminus\cutz_{\cang'} \times \complex \times \complex\setminus\cut{\cang}$ in $(z,\,s,\,w)$.\\
The domains for $\Phi(e^{2\pi ia},\,1-s,\,w)$, $\Phi(e^{-2\pi iw},\,s,\,a)$, and $\Phi(e^{2\pi iw},\,s,\,1-a)$ are
\begin{align*}
\begin{array}{l}
\dom{1} e^{2\pi ia} \notin [1,\,\infty),\,1-s \in \complex,\,w \notin \cut{\cang},\\
\dom{2} e^{-2\pi iw} \notin [1,\,\infty),\,s \in \complex,\,a \notin \cut{\cang},\\
\dom{3} e^{2\pi iw} \notin [1,\,\infty),\,s \in \complex,\,1-a \notin \cut{\cang},
\end{array}
\end{align*}
respectively.
The condition $e^{2\pi ia} \in [1,\,\infty)$ holds if and only if $a \in \rint+i\real_{\leq 0}$.
Thus the intersection in $(a,\,s,\,w)$ is
 $a \in \complex\setminus(\cutsym{\cang} \cup (\rint+i\real_{\leq 0}))$,
 $s \in \complex$, $w \in \complex\setminus(\cut{\cang} \cup (\rint+i\real))$.

\noindent
\eqref{finaldomains}(2) On the domain $\complex\setminus\cutz_{\cang'} \times \complex$ in $(z,\,s)$ with fixed $w \in \cut{\cang}\setminus\rint_{\leq 0}$.\\
The domains for $\Phi(e^{2\pi ia},\,1-s,\,w)$, $\Phi(e^{-2\pi iw},\,s,\,a)$, and $\Phi(e^{2\pi iw},\,s,\,1-a)$ are
\begin{align*}
\begin{array}{l}
\dom{1} e^{2\pi ia} \notin [1,\,\infty),\,1-s \in \complex \text{ with fixed } w \in \cut{\cang}\setminus\rint_{\leq 0},\\
\dom{2} e^{-2\pi iw} \notin [1,\,\infty),\,s \in \complex \text{ with fixed } a \in \cut{\cang}\setminus\rint_{\leq 0},\\
\dom{3} e^{2\pi iw} \notin [1,\,\infty),\,s \in \complex \text{ with fixed } 1-a \in \cut{\cang}\setminus\rint_{\leq 0},
\end{array}
\end{align*}
respectively.
The condition $a,\,1-a \in \cut{\cang}\setminus\rint_{\leq 0}$ implies $\cang \in 2\pi\rint$,
 which contradicts to the condition $\cang \in \real\setminus 2\pi\rint$.

\noindent
\eqref{finaldomains}(3) On the domain $\complex\setminus P_z \times \complex\setminus\cut{\cang}$ in $(s,\,w)$ with fixed $z \in \cutz_{\cang'}$.\\
The domains for $\Phi(e^{2\pi ia},\,1-s,\,w)$, $\Phi(e^{-2\pi iw},\,s,\,a)$, and $\Phi(e^{2\pi iw},\,s,\,1-a)$ are
\begin{align*}
\begin{array}{l}
\dom{1} 1-s \in \complex\setminus P_{e^{2\pi ia}},\,w \notin \cut{\cang},\,e^{2\pi ia} \in [1,\,\infty) \text{ with fixed } a \in \complex,\\
\dom{2} s \in \complex\setminus P_{e^{-2\pi iw}},\,a \notin \cut{\cang},\,e^{-2\pi iw} \in [1,\,\infty) \text{ with fixed } w \in \complex,\\
\dom{3} s \in \complex\setminus P_{e^{2\pi iw}},\,1-a \notin \cut{\cang},\,e^{2\pi iw} \in [1,\,\infty) \text{ with fixed } w \in \complex,
\end{array}
\end{align*}
respectively.
The condition $e^{2\pi ia},\,e^{\pm2\pi iw} \in [1,\,\infty)$ is equivalent to $a \in \rint+i\real_{\leq 0}$, $w \in \rint$.
Since $w,\,a,\,1-a \notin \cut{\cang} \supset \rint_{\leq 0}$, we have $w \in \rint_{>0}$, $a \notin \rint$.
Thus the intersection in $(a,\,s,\,w)$ is $s \in \complex\setminus \{1\}$ with fixed
 $a \in (\rint+i\real_{\leq 0})\setminus\cutsym{\cang}$, $w \in \rint_{>0}$.

\noindent
\eqref{finaldomains}(4) On the domain $\complex\setminus P_z$ in $s$ with fixed $z \in \cutz_{\cang'}$, $w \in \cut{\cang}\setminus\rint_{\leq 0}$.\\
The domains for $\Phi(e^{2\pi ia},\,1-s,\,w)$, $\Phi(e^{-2\pi iw},\,s,\,a)$, and $\Phi(e^{2\pi iw},\,s,\,1-a)$ are
\begin{align*}
\begin{array}{l}
\dom{1} 1-s \in \complex\setminus P_{e^{2\pi ia}},\,e^{2\pi ia} \in [1,\,\infty) \text{ with fixed $a \in \complex$, $w \in \cut{\cang}\setminus\rint_{\leq 0}$},\\
\dom{2} s \in \complex\setminus P_{e^{-2\pi iw}},\,e^{-2\pi iw} \in [1,\,\infty) \text{ with fixed $w \in \complex$, $a \in \cut{\cang}\setminus\rint_{\leq 0}$},\\
\dom{3} s \in \complex\setminus P_{e^{2\pi iw}},\,e^{2\pi iw} \in [1,\,\infty) \text{ with fixed $w \in \complex$, $1-a \in \cut{\cang}\setminus\rint_{\leq 0}$},
\end{array}
\end{align*}
respectively.
The condition $a,\,1-a \in \cut{\cang}\setminus\rint_{\leq 0}$ implies $\cang \in 2\pi\rint$,
 which contradicts to the condition $\cang \in \real\setminus 2\pi\rint$.

Summarizing them with the identity theorem for holomorphic functions,
 we have the statement of the proposition for Lerch's equation.
We omit to give a proof for Apostol's equation because it is the same manner as above.
\end{proof}

\begin{remark}
In the context of the proposition above,
Lerch chose the cut $\cut{\cang}$ with angle $\cang=-\pi$ in the proof of Lerch's equation~\cite{Ler}
 (In this case $\cut{\cang}$ is exactly the negative real axis together with $0$).
Also, Apostol seems to have chosen the cut with angle $\cang=-\pi$, but not explicitly mentioned in~\cite{Apo}.
\end{remark}

\subsection{Proof of Theorem~\ref{thm:equivalence}}

We first show that if Apostol's equation~\eqref{eq:Apostol} holds on the domain~\eqref{domain:equiv} ($=$\eqref{domain:Apostol})
 then Lerch's equation~\eqref{eq:Lerch} holds on the same domain,
 which was originally proved by Apostol~\cite{Apo} for a smaller region of $a,\,w$.
Let
\begin{align*}
\Omega(s)&:=2(2\pi)^{-s}\left(\cos\frac{\pi s}{2}\right)\Gamma(s),\\
\Lambda^-(a,\,s,\,w)&:=\Phi(e^{2\pi ia},\,s,\,w)-e^{-2\pi ia}\Phi(e^{-2\pi ia},\,s,\,1-w).
\end{align*}
By Proposition~\ref{prop:LA}, Apostol's equation is differentiable in $a,\,s$, and $w$ on the domain~\eqref{domain:equiv},
 and thus $\Lambda^-(a,\,s,\,w)=2\Phi(e^{2\pi ia},\,s,\,w)-\Lambda(a,\,s,\,w)$ is also differentiable.
It follows from Lemma~\ref{lem:dif} that
\begin{align}\label{eq:dif}
&\frac{\partial}{\partial a}\Lambda(a,\,1-s,\,w)
=2\pi i\left\{\Lambda^-(a,\,-s,\,w)-w\Lambda(a,\,1-s,\,w)\right\}, \notag\\
&\frac{\partial}{\partial a}\Lambda(-w,\,s,\,a)
=-s\Lambda^-(-w,\,s+1,\,a),\\
&\frac{\partial}{\partial a}e^{-2\pi iaw}\Omega(s)\Lambda(-w,\,s,\,a)
=-e^{-2\pi iaw}\Omega(s) \notag\\
&\hspace{7em} \times\left\{s\Lambda^-(-w,\,s+1,\,a)+2\pi iw\Lambda(-w,\,s,\,a)\right\}. \notag
\end{align}
Using Apostol's equation $\Lambda(a,\,1-s,\,w)=e^{-2\pi iaw}\Omega(s)\Lambda(-w,\,s,\,a)$ yields
\begin{align*}
\Lambda^-(a,\,-s,\,w)
=\frac{is}{2\pi}e^{-2\pi iaw}\Omega(s)\Lambda^-(-w,\,s+1,\,a).
\end{align*}
By the formula $\Gamma(s+1)=s\Gamma(s)$ and the substitution $s \to s-1$, we have the functional equation
\begin{align}\label{eq:Apostol-}
\Lambda^-(a,\,1-s,\,w)=2i(2\pi)^{-s}\left(\sin\frac{\pi s}{2}\right)\Gamma(s)e^{-2\pi iaw}\Lambda^-(-w,\,s,\,a),
\end{align}
which was originally given by Apostol~\cite[p.164]{Apo} for a small region of $a,\,w$
 (However, it seems that there is a misprint: not ``$\exp(-2\pi ia(1-x))$" but ``$\exp(2\pi ia(1-x))$" in the functional relation in~\cite[p.164]{Apo}).
Then, adding two equations~\eqref{eq:Apostol}, \eqref{eq:Apostol-}
 gives Lerch's equation~\eqref{eq:Lerch}.

Conversely, assume that Lerch's equation~\eqref{eq:Lerch} holds on the domain~\eqref{domain:equiv} ($=$ \eqref{domain:Lerch} by taking $U_a=U_w=U$).
Using Lerch's equation~\eqref{eq:Lerch} with the substitution
 $(a,\,s,\,w) \to (1-a,\,s,\,1-w)$ (the domain~\eqref{domain:equiv} is invariant under this substitution) yields
\begin{align*}
&\frac{(2\pi)^s}{\Gamma(s)}e^{2\pi iaw}\Lambda(a,\,1-s,\,w)\\
&\hspace{1em} =\frac{(2\pi)^s}{\Gamma(s)}e^{2\pi iaw}\big\{\Phi(e^{2\pi ia},\,1-s,\,w)+e^{-2\pi ia}\Phi(e^{2\pi i(1-a)},\,1-s,\,1-w)\big\}\\
&\hspace{1em} =e^{\pi i\frac{s}{2}}\Phi(e^{-2\pi iw},\,s,\,a)+e^{\pi i(-\frac{s}{2}+2w)}\Phi(e^{2\pi iw},\,s,\,1-a)\\
&\hspace{5em} +e^{-\pi i\frac{s}{2}}\Phi(e^{2\pi i(1-w)},\,s,\,a)+e^{\pi i(\frac{s}{2}+2w)}\Phi(e^{-2\pi i(1-w)},\,s,\,1-a)\\
&\hspace{1em} =(e^{\pi i\frac{s}{2}}+e^{-\pi i\frac{s}{2}})\left\{\Phi(e^{-2\pi iw},\,s,\,a)+e^{2\pi iw}\Phi(e^{2\pi iw},\,s,\,1-a)\right\}\\
&\hspace{1em} =2\left(\cos\frac{\pi s}{2}\right)\Lambda(-w,\,s,\,a).
\end{align*}
This is Apostol's equation~\eqref{eq:Apostol}.

\begin{lemma}\label{lem:dif}
The differential-difference equations
\begin{align*}
&\frac{\partial}{\partial a}\Phi (e^{2\pi ia},\,1-s,\,w)
=2\pi i\big\{\Phi(e^{2\pi ia},\,-s,\,w)-w\Phi(e^{2\pi ia},\,1-s,\,w)\big\},\\
&\frac{\partial}{\partial a}\Phi (e^{-2\pi ia},\,1-s,\,1-w)
=-2\pi i\big\{\Phi(e^{-2\pi ia},\,-s,\,1-w)\\
&\hspace{14em} +(w-1)\Phi(e^{-2\pi ia},\,1-s,\,1-w)\big\},\\
&\frac{\partial}{\partial a}\Phi (e^{-2\pi iw},\,s,\,a)
=-s\Phi(e^{-2\pi iw},\,s+1,\,a),\\
&\frac{\partial}{\partial a}\Phi (e^{2\pi iw},\,s,\,1-a)
=s\Phi(e^{2\pi iw},\,s+1,\,1-a)
\end{align*}
 hold for all $(a,\,s,\,w) \in$ \eqref{domain:equiv}.
\end{lemma}

\begin{proof}
We can differentiate the series~\eqref{def:init} by terms
 on the the first domain of~\eqref{initdomains} with the substitution $(z,\,s,\,w) \to (e^{2\pi ia},\,1-s,\,w)$.
$$\frac{\partial}{\partial a}\Phi (e^{2\pi ia},\,1-s,\,w)
=2\pi i\big\{\Phi(e^{2\pi ia},\,-s,\,w)-w\Phi(e^{2\pi ia},\,1-s,\,w)\big\}.$$
By Theorem~\ref{thm:anaconti} and the identity theorem for holomorphic functions,
 this equation holds for all $(a,\,s,\,w) \in \complex\setminus(\rint+i\real_{\leq 0}) \times \complex \times \complex\setminus\cut{\cang}$,
 which contains \eqref{domain:equiv}.
Similarly, we have the other equations in the statement of the lemma.
\end{proof}

\begin{remark}
By the substitutions $(a,\,s,\,w) \to (\mp w,\,1-s,\,\pm a)$ in the equations~\eqref{eq:dif},
 we obtain the differential-difference equations
\begin{align*}
&\frac{\partial}{\partial w}\Lambda(-w,\,s,\,a)
=-2\pi i\left\{\Lambda^-(-w,\,s-1,\,a)-a\Lambda(-w,\,s,\,a)\right\},\\
&\frac{\partial}{\partial w}\Lambda(a,\,1-s,\,w)=(s-1)\Lambda^-(a,\,2-s,\,w).
\end{align*}
In the same manner as above, we have
\begin{align*}
&\frac{\partial}{\partial a}\Lambda^-(-w,\,s,\,a)
=-s\Lambda(-w,\,s+1,\,a),\\
&\frac{\partial}{\partial a}\Lambda^-(a,\,1-s,\,w)
=2\pi i\left\{\Lambda(a,\,-s,\,w)-w\Lambda^-(a,\,1-s,\,w)\right\},\\
&\frac{\partial}{\partial w}\Lambda^-(a,\,1-s,\,w)
=(s-1)\Lambda(a,\,2-s,\,w),\\
&\frac{\partial}{\partial w}\Lambda^-(-w,\,s,\,a)
=-2\pi i\left\{\Lambda(-w,\,s-1,\,a)-a\Lambda^-(-w,\,s,\,a)\right\}.
\end{align*}
One can deduce from these calculations
 that we cannot find new functional equations by the differentiation in $a,\,w$
 of Apostol's equation (and Lerch's equation).
\end{remark}





\end{document}